\documentclass[submission,copyright]{eptcs}
\usepackage{url}
\usepackage{amsfonts}
\usepackage{amsmath}
\usepackage{amsthm}
\usepackage{amssymb}
\usepackage{latexsym}
\usepackage{graphicx}
\usepackage{graphicx}
\usepackage{xspace}
\usepackage{verbatim}
\usepackage{ifthen}
\usepackage{hyperref}

\usepackage{tikz}
\usepackage{tikzfig}    
\usepackage{tikzcdiag}  

\usetikzlibrary{trees}
\usetikzlibrary{topaths}
\usetikzlibrary{decorations.pathmorphing}
\usetikzlibrary{decorations.markings}
\usetikzlibrary{snakes,matrix,backgrounds,folding}
\usetikzlibrary{chains,scopes,positioning,fit}
\usetikzlibrary{arrows,shadows}
\usetikzlibrary{calc}
\usetikzlibrary{chains}
\usetikzlibrary{shapes,shapes.geometric,shapes.misc}

\pgfsetlayers{background,edgelayer,nodelayer,main}

\tikzstyle{small map edge}=[|->, thick]

\tikzstyle{string graph}=[scale=0.6]
\tikzstyle{sg diredge}=[-stealth]
\tikzstyle{wire label}=[font=\footnotesize, auto]
\tikzstyle{sg vertex}=[circle,minimum width=2.2mm,fill=white,draw=black,inner sep=0mm]
\tikzstyle{labelled sg vertex}=[circle,minimum width=7mm,fill=white,draw=black,inner sep=0mm]
\tikzstyle{sg grey vertex}=[sg vertex,fill=gray!30!white]
\tikzstyle{sg black vertex}=[sg vertex,fill=black]
\tikzstyle{sg bold vertex}=[circle,minimum width=2.2mm,fill=white,draw=black,very thick,inner sep=0mm]
\tikzstyle{sg wire vertex}=[circle,minimum width=1mm,fill=black,inner sep=0mm]

\tikzstyle{bbox edge}=[draw=blue]
\tikzstyle{bbox include}=[->,draw=blue]
\tikzstyle{bbox corner}=[inner sep=0pt,rectangle,fill=blue,draw=blue,minimum width=1.5mm,minimum height=1.5mm]

\tikzstyle{dotpic}=[scale=0.6]
\tikzstyle{dot}=[inner sep=0.7mm,minimum width=0pt,minimum height=0pt,fill=black,draw=black,shape=circle]
\tikzstyle{white dot}=[dot,fill=white]
\tikzstyle{alt white dot}=[white dot,label={[xshift=2.9mm,yshift=-0.1mm]left:$\cdot$}]
\tikzstyle{gray dot}=[dot,fill=gray!50]

\tikzstyle{diredge}=[->]
\tikzstyle{braceedge}=[decorate,decoration={brace,amplitude=2mm,raise=-1mm}]
\tikzstyle{small braceedge}=[decorate,decoration={brace,amplitude=1mm,raise=-1mm}]

\tikzstyle{square box}=[rectangle,fill=white,draw=black,minimum height=6mm,minimum width=6mm]

\newcommand{\catSet}{\ensuremath{\mathbf{Set}}\xspace}

\newcommand{\catSGraph}{\ensuremath{\mathbf{SGraph}}\xspace}
\newcommand{\catSPatGraph}{\ensuremath{\mathbf{SPatGraph}}\xspace}
\newcommand{\catGraph}{\ensuremath{\mathbf{Graph}}\xspace}

\newcommand{\slicecat}[2]{#1 / #2}
\newcommand{\catGraphSlice}{\ensuremath{\slicecat{\catGraph}{\mathcal{G}_2}}}

\DeclareMathOperator{\DROP}{DROP}
\DeclareMathOperator{\KILL}{KILL}
\DeclareMathOperator{\MERGE}{MERGE}
\DeclareMathOperator{\COPY}{COPY}
\DeclareMathOperator{\PDROP}{PDROP}
\DeclareMathOperator{\PKILL}{PKILL}
\DeclareMathOperator{\PMERGE}{PMERGE}
\DeclareMathOperator{\PCOPY}{PCOPY}

\newlength{\hookrightarrowwidth}
\DeclareMathOperator{\matches}{%
\makebox[\hookrightarrowwidth][c]{$\hookrightarrow$}
\hspace*{-\hookrightarrowwidth}%
\makebox[\hookrightarrowwidth][c]{\raise0.5mm\hbox{$\, ^{\sim}$}}%
}

\DeclareMathOperator{\leftmatches}{%
\makebox[\hookrightarrowwidth][c]{$\hookleftarrow$}
\hspace*{-\hookrightarrowwidth}%
\makebox[\hookrightarrowwidth][c]{\raise0.5mm\hbox{$\, ^{\sim}$}}%
}

\DeclareMathOperator{\notleftmatches}{%
\makebox[\hookrightarrowwidth][c]{$\hookleftarrow$}
\hspace*{-\hookrightarrowwidth}%
\makebox[\hookrightarrowwidth][c]{\raise0.5mm\hbox{$\, ^{\sim}$}}%
\hspace*{-\hookrightarrowwidth}%
\makebox[\hookrightarrowwidth][c]{\raise-0.2mm\hbox{\footnotesize{$/$}}}%
}

\newcommand{\graphminus}{\setminus}
\newcommand{\emptygraph}{\varnothing}

\newcommand{\In}{\textrm{In}}
\newcommand{\Out}{\textrm{Out}}
\newcommand{\Bound}{\textrm{Bound}}

\newcommand{\cmdrewritesto}{\tikz[baseline=-0.25em] { \draw [-open triangle 45, line width=0.2pt] (0,0) -- (0.5,0); }\,}

\DeclareMathOperator{\rewritesto}{\cmdrewritesto}

\theoremstyle{plain} 
\newtheorem{theorem}{Theorem}[section]
\newtheorem{proposition}[theorem]{Proposition}
\newtheorem{lemma}[theorem]{Lemma}

\theoremstyle{definition}
\newtheorem{definition}[theorem]{Definition}
\newtheorem{definitions}[theorem]{Definitions}

\newtheorem{example}[theorem]{Example}

\newcommand{\nulp}[1]{}

\providecommand{\event}{DCM 2012} 

\title{Pattern graph rewrite systems}
\def\titlerunning{Pattern graph rewrite systems}

\author{
Aleks Kissinger
\institute{Department of Computer Science\\
University of Oxford, United Kingdom}
\email{alexander.kissinger@cs.ox.ac.uk}
\and
Alex Merry
\institute{Department of Computer Science\\
University of Oxford, United Kingdom}
\email{alex.merry@cs.ox.ac.uk}
\and
Matvey Soloviev
\institute{Computer Laboratory\\
University of Cambridge, United Kingdom}
\email{ms900@cam.ac.uk}
}
\def\authorrunning{A.\ Kissinger, A.\ Merry \& M.\ Soloviev}

\date{Draft: \today}

\begin{document}

\maketitle

\begin{abstract}
String diagrams are a powerful tool for reasoning about physical processes,
logic circuits, tensor networks, and many other compositional structures.
Dixon, Duncan and Kissinger introduced \emph{string graphs}, which are a
combinatoric representations of string diagrams, amenable to automated
reasoning about diagrammatic theories via graph rewrite systems. In this
extended abstract, we show how the power of such rewrite systems can be
greatly extended by introducing \emph{pattern graphs}, which provide a means
of expressing infinite families of rewrite rules where certain marked
subgraphs, called $!$-boxes (``bang boxes''), on both sides of a rule can be
copied any number of times or removed. After reviewing the string graph
formalism, we show how string graphs can be extended to pattern graphs and how
pattern graphs and pattern rewrite rules can be instantiated to concrete
string graphs and rewrite rules. We then provide examples demonstrating the
expressive power of pattern graphs and how they can be applied to study
interacting algebraic structures that are central to categorical quantum
mechanics.
\end{abstract}

\section{Introduction}
\label{sec:intro}

String diagrams consist of a collection of \textit{boxes} representing processes with some inputs and outputs, and \textit{wires}, representing the composition of these processes.
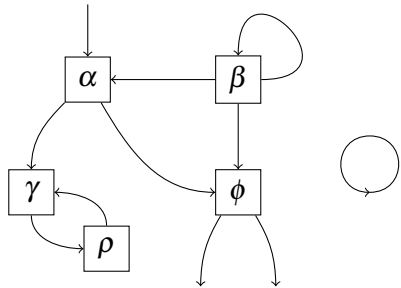
\ctikzfig{kissinger_tensor_diagram}

They were introduced by Penrose in 1971 to describe (abstract) tensor networks~\cite{Penrose1971}, but were later shown to be a much more general tool for expressing morphisms in arbitrary monoidal categories. Joyal and Street showed in 1991 that string diagrams could be formalised as topological graphs carrying extra structure and used to construct \textit{free} (symmetric, braided, traced, etc.) monoidal categories~\cite{JS}. As such, they are a powerful tool for reasoning about algebraic structures \textit{internal} to monoidal categories, like those employed by Abramsky and Coecke's program of \textit{categorical quantum mechanics}~\cite{AC2004,Coecke2008,CoeckeKissinger2010,Coecke:2008p681}.

However, while they provide an intuitive, geometric notion of a composed
process, topological graphs are unwieldy to manipulate by computer program. To
solve this problem, Dixon, Duncan and Kissinger introduced a discrete version
of string diagrams, called \textit{string graphs}~\cite{Dixon2010}. The key
difference is that ``wires'', which in the Joyal and Street construction are
represented by copies of the real interval $[0,1]$, are replaced by chains of
special vertices called \textit{wire-vertices}.
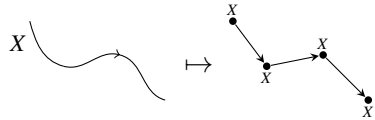
\ctikzfig{kissinger_wire}

Using string graphs, we can reason about algebraic structures in monoidal categories automatically using double-pushout graph rewriting~\cite{Ehrig1973}. This translation allows many techniques to be imported with very little change from term rewriting literature into the study of graphical calculi. However, in the course of applying graph rewrite systems, there are certain circumstances where a finite set of graph rewrite rules does not suffice. For instance, in~\cite{Coecke2008} the authors focused on the study of how classical data (in this case, data associated with measurement outcomes) propagates through a quantum system. This relies crucially on so-called ``spiders''. The distinguishing feature they highlighted about classical, as opposed to quantum, data is that it can be freely created, compared, copied, or deleted. They represent any combination of these operations as a \textit{spider}, with a crucial identity, called the \textit{spider law}, which says that connected spiders fuse together.
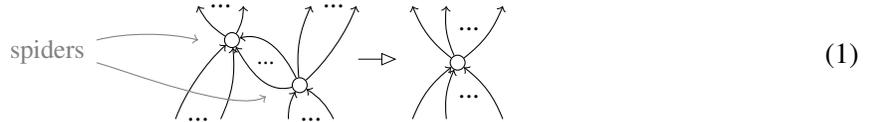
\begin{equation}\label{eq:spider-comp}
  \begin{tikzpicture}[string graph]
	\begin{pgfonlayer}{nodelayer}
		\node [style=none] (0) at (-4, 1.25) {};
		\node [style=none] (1) at (-3.5, 1.25) {...};
		\node [style=none] (2) at (-3, 1.25) {};
		\node [style=none] (3) at (-1.5, 1.25) {};
		\node [style=none] (4) at (-1, 1.25) {...};
		\node [style=none] (5) at (-0.5, 1.25) {};
		\node [style=none] (6) at (0.75, 1.25) {};
		\node [style=none] (7) at (1.5, 1.25) {};
		\node [style=none] (8) at (2.75, 1.25) {};
		\node [style=none] (9) at (2, 0.75) {...};
		\node [style=white dot] (10) at (-3.25, 0.5) {};
		\node [style=none, font=\footnotesize] (11) at (-2.5, 0) {...};
		\node [style=none] (12) at (0, 0) {$\rewritesto$};
		\node [style=white dot] (13) at (1.75, 0) {};
		\node [style=white dot] (14) at (-1.75, -0.5) {};
		\node [style=none] (15) at (2, -0.75) {...};
		\node [style=none] (16) at (-4.5, -1.25) {};
		\node [style=none] (17) at (-4, -1.25) {...};
		\node [style=none] (18) at (-3.5, -1.25) {};
		\node [style=none] (19) at (-2, -1.25) {};
		\node [style=none] (20) at (-1.5, -1.25) {...};
		\node [style=none] (21) at (-1, -1.25) {};
		\node [style=none] (22) at (0.75, -1.25) {};
		\node [style=none] (23) at (1.5, -1.25) {};
		\node [style=none] (24) at (2.75, -1.25) {};
		\node [style=none] (25) at (-6.25, 0.5) {};
		\node [style=none] (26) at (-4, 0.5) {};
		\node [style=none] (27) at (-2.5, -0.75) {};
		\node [style=none] (28) at (-6.25, 0) {};
		\node [style=none, anchor=east] (29) at (-6.5, 0.25) {\small\color{gray}spiders};
	\end{pgfonlayer}
	\begin{pgfonlayer}{edgelayer}
		\draw [style=diredge, bend left=15] (22.center) to (13);
		\draw [style=diredge, bend left=15] (19.center) to (14);
		\draw [style=diredge, bend right=15] (10) to (2.center);
		\draw [style=diredge, in=-60, out=-165] (14) to (10);
		\draw [style=diredge, in=15, out=116] (14) to (10);
		\draw [style=diredge, bend right=15] (18.center) to (10);
		\draw [style=diredge, bend left=15] (14) to (3.center);
		\draw [style=diredge, bend right=15] (13) to (8.center);
		\draw [style=diredge, bend right=15] (21.center) to (14);
		\draw [style=diredge, bend left=15] (23.center) to (13);
		\draw [style=diredge, bend right=15] (24.center) to (13);
		\draw [style=diredge, bend left=15] (10) to (0.center);
		\draw [style=diredge, bend right=15, looseness=0.75] (14) to (5.center);
		\draw [style=diredge, bend left=15, looseness=0.75] (16.center) to (10);
		\draw [style=diredge, bend left=15] (13) to (6.center);
		\draw [style=diredge, bend left=15] (13) to (7.center);
		\draw [style=diredge, draw=gray, bend left=15, looseness=0.75] (25.center) to (26.center);
		\draw [style=diredge, draw=gray, in=-150, out=-16, looseness=0.50] (28.center) to (27.center);
	\end{pgfonlayer}
\end{tikzpicture}}
\end{equation}

This rule succinctly sums up an infinite family of rules, namely one for every arity of the two spiders involved. However, the use of ellipses is part of the meta-language, rather than the diagram itself. What we aim to do is replace this informal notion with diagrammatic syntax. We do this by introducing \textit{pattern graphs}. Pattern graphs contain one or more labelled subgraphs called $!$-boxes. To instantiate a pattern graph, the contents of its $!$-boxes (along with any edges in or out) can be copied 0 or more times. So, a single pattern graph represents an infinite family of concrete graphs.
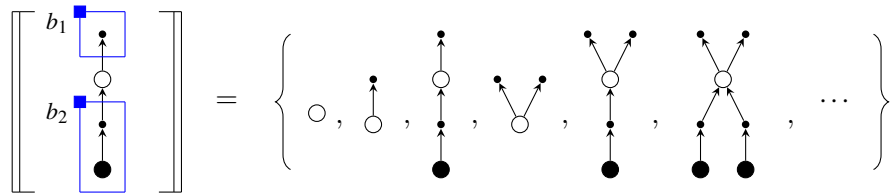
\ctikzfig{kissinger_bang_graph_ex}

If two pattern graphs have coinciding !-boxes, we can form them into pattern rewrite rules. For instance, the spider law can be reformulated:
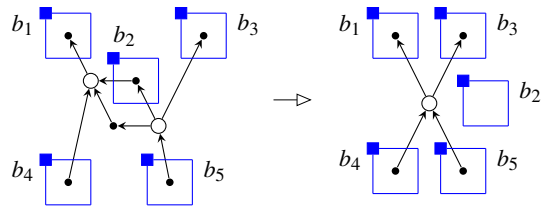
\ctikzfig{kissinger_spiderpattern}

This presents \eqref{eq:spider-comp} in a manner that is machine-readable. Also note that in the process of formulating this rule, we have removed an ambiguity on the LHS. Namely, we wish to have zero or more wires as inputs and outputs to the two spiders, yet we need one or more wires connecting the two spiders for the equation to hold.

Dixon and Duncan have previously~\cite{dixon2009} introduced a notion of
pattern graphs using !-boxes. However, the underlying graph formalism, which
did without (internal) wire-vertices, was ill-behaved with respect to the
interpretation of the graphs as morphisms in a monoidal category. This
extended abstract extends that work in three important ways. Firstly, it
formalises the notions of pattern graph, pattern graph instantiation, and
pattern rewriting in the context of string graphs, which were proven in
\cite{Dixon2010} to be sound and complete with respect to their interpretation
as morphisms in monoidal categories. Secondly, it shows that the latter two
operations are sound and consistent with respect to the interpretation of
string graphs as morphisms in a monoidal category. Thirdly, it extends Dixon
and Duncan's origin notion of a pattern graph by allowing edges to be repeated
(via wire-vertices in !-boxes) and it increases the expressiveness of the
language by allowing !-boxes to nest and overlap. This allows the expression
of previously unexpressible equivalences, such as the \emph{path-counting}
rule,
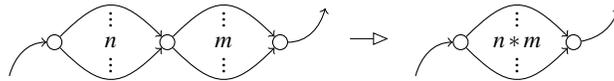
\ctikzfig{kissinger_pathcounting}
which can now be formalised as follows:
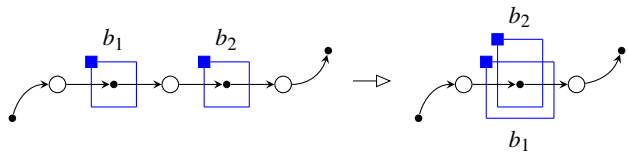
\ctikzfig{kissinger_pathcounting_b}

The rest of the paper is structured as follows. In section~\ref{sec:string-graphs}, we briefly review the category of string graphs. In section~\ref{sec:pattern-graphs}, we define pattern graphs and the method by which pattern graphs can be instantiated to concrete graphs. In section~\ref{sec:rw-pattern-graphs}, we show how this can be extended to pattern graph rewrite rules and show how pattern rules can be matched and applied to concrete string graphs. Finally, we conclude and discuss future work in section~\ref{sec:conclusion}.

\section{Related work} 
\label{sec:related}

As already mentioned, this work improves upon the specification of $!$-boxes
in~\cite{dixon2009}. The original inspiration for the term ``$!$-box'' in that
paper is the ``bang'' operation from classical linear logic (CLL) introduced
by Girard~\cite{Girard1987}. Its interpretation in that context is a logical
expression that can be ``consumed'' any number of times in the course of the
proof.

Lafont introduced an alternative, and more flexible, 2-dimensional
calculus~\cite{Lafont95}.  It does not rely on symmetry, or on traced or
compact structure, but this also makes it harder to work with as these
properties allow us to do genuine graph rewriting.

Researchers at Twente introduced two ways by which richer families of graphs could be matched and rewritten using something akin to pattern graphs. The first method, initiated by Rensink, uses quantified graph transformation rules, where subgraphs are attached to a tree of alternating quantifiers~\cite{Rensink2006,Rensink2009}. Unlike the transformation rules we consider, this method allows matchings to be non-full on all vertices in a pattern graph, so an edge in the pattern can be interpreted as an existentially-quantified statement on the attached subgraph, rather than a requirement that all incident edges must be matched. Rensink showed that such statements could be generalised to include negations, universals, and nested quantifiers.

The second method takes inspiration from abstraction/refinement-style model checking. Using graph abstraction~\cite{Boneva2007}, large or infinite families of graphs can be represented using coarse-grained abstract graphs. While this often has the side-effect of producing abstract graphs that match many more graphs than those of interest, it has the useful property that any high-level properties proven about the abstract graph hold for any concrete graph it represents.

Both of these methods are implemented on the GROOVE platform, which is a general-purpose graph rewriting tool geared toward model-checking~\cite{GROOVE}.

\section{The category of string graphs} 
\label{sec:string-graphs}

We recall the definition of \emph{string graphs}, introduced using the name open-graphs in
\cite{Dixon2010}.

String diagrams can have wires that are not connected to vertices at one or both ends and wires that are connected to themselves to form circles. As we mentioned in section~\ref{sec:intro}, we cope with these situations by replacing wires with chains of special place-holder vertices called \textit{wire-vertices}. The other type of vertices in a string graph are called \textit{node-vertices}, which should be considered the ``logical'' vertices of a diagram, and are used to represent some operation, process, or morphism.
We now provide some basic definitions in order to fix graph notation.

\begin{definition}
  Let $\catGraph$ be the category of graphs. It is defined as the functor
  category $[\mathbb G, \catSet]$, for $\mathbb G$ defined as:
  \begin{center}
	\cpair{E}{V}{s}{t}
  \end{center}
  $E$ identifies the edges of the graph, and $V$ the vertices.
  $s$ and $t$ are functions taking an edge to its source and target
  respectively.
\end{definition}

If $t(e) = v$ then $e$ is called an \emph{in-edge} of $v$ and if $s(e) = v$ then $e$ is called an \emph{out-edge} of $v$. If $v'$ is the target of one of the out-edges of $v$, it is called a \textit{successor} of $v$. Similarly, if $v'$ is the source of one of the in-edges of $v$, it is called a \textit{predecessor} of $v$. We denote the set of all successors and predecessors for a given vertex $v$ as $\textrm{succ}(v)$ and $\textrm{pred}(v)$, respectively.

We shall often make use of the graph-theoretic subtraction. For a subgraph $H$
of $G$, let $G \graphminus H$ be the largest subgraph of $G$ that is disjoint
from $H$.

The typegraph $\mathcal G_2$ will be used to distinguish node-vertices from wire-vertices.
\ctikzfig{kissinger_typegraph_g2}

\begin{definition}[$\catSGraph$]
The category $\catSGraph$ of \emph{string graphs} is the full subcategory of
the slice category $\catGraphSlice$ induced by the objects where
each wire-vertex has at most one in-edge and one out-edge.
\end{definition}

This slice construction allows string graphs to be represented as graphs with a
typing morphism to $\mathcal G_2$. We refer to a single chain of wire-vertices
as a \textit{wire}. The slice construction also ensures that every path
between two node-vertices must be connected by a wire containing at least one wire-vertex. This is important both for the concept of matching and for the case where the wire-vertex carries type information about the wire.

\begin{example}
  A diagrammatic presentation of a string graph:
  \ctikzfig{kissinger_sgraph_ex}
\end{example}

\begin{definitions}[$\catSGraph$ Notation]
  If a wire-vertex has no in-edges, it is called an \emph{input}. We write the set of inputs of a string graph $G$ as $\In(G)$. Similarly, a wire-vertex with no out-edges is called an \emph{output}, and the set of outputs is written $\Out(G)$. The inputs and outputs define a string graph's \emph{boundary}, $\Bound(G) := \In(G) + \Out(G)$. If a boundary point has no in-edges and no out-edges, (it is both and input and output) it is called an \emph{isolated point}. A string graph consisting of only isolated points is called a \emph{point-graph}.
\end{definitions}

\nulp{\begin{fullproof}
\begin{lemma}
	\label{lemma:sg-morphism-bounds}
	If $f:G \rightarrow H$ is a morphism in $\catSGraph$ and $v$ is a vertex in
	$G$ that maps to an input (respectively output) of $H$ under $f$, then $v$
	must be an input (respectively output).
\end{lemma}
\begin{proof}
	Suppose $e$ is an edge of $G$ such that $t_G(e) = v$.  Then
	$t_H(f(e))$ must be $f(v)$, and so $f(v)$ cannot be an input unless
	$v$ is.  The output case is symmetric.
\end{proof}
\end{fullproof}}

These definitions can be easily extended to handle multiple
node-vertex and wire types by using a richer typegraph. In general, one can
turn any monoidal signature $T$ into a typegraph $\mathcal G_T$ and use
$\mathcal G_T$-typed graphs to construct the free (traced symmetric) monoidal
category over the signature $T$. For details, see~\cite{DixonKissinger2010}
or~\cite{aleksThesis}. However, for the main ideas in the coming sections, it
suffices to consider string graphs with a single node-vertex and wire type.

\section{Pattern graphs and instantiation} 
\label{sec:pattern-graphs}

Before proceeding to the notion of $!$-boxes, it is useful to first define an
\textit{open subgraph} of a string graph. Intuitively, these are full
subgraphs that contain only complete wires. One way to say this is the
graph-theoretic subtraction does not create any new boundaries.

\begin{definition}
	A subgraph $O$ of a string graph $G$ is said to be \textit{open} if
	$\In(G\graphminus O) \subseteq \In(G)$ and
	$\Out(G\graphminus O) \subseteq \Out(G)$.
\end{definition}

\begin{comment}
\begin{proposition}
	A subgraph $O$ of a string graph $G$ is open if and only if
	$\In(G\graphminus O) = \In(G)\graphminus\In(O)$ and
	$\Out(G\graphminus O) = \Out(G)\graphminus\In(O)$.
\end{proposition}
\nulp{\begin{fullproof}
\begin{proof}
	Suppose $i \in \In(G)$.  If it is in $O$, it must also be in $\In(O)$,
	so if it is in $\In(G)\graphminus\In(O)$, it cannot be in $O$ at all.
	But then it is in $G\graphminus O$, which is a subgraph of $G$, and it
	is an input of $G$, so it must be an input of $G\graphminus O$.  So
	$\In(G)\graphminus\In(O) \subseteq \In(G\graphminus O)$.

	Now suppose $i \in \In(G\graphminus O)$.  Then $i$ is not in $O$, and so
	not in $\In(O)$.  So $i$ is in $\In(G)\graphminus\In(O)$ if and only if
	$i$ is in $\In(G)$, and hence $\In(G\graphminus O) =
	\In(G)\graphminus\In(O)$ if and only if $\In(G\graphminus O) \subseteq
	\In(G)$.
	
	The case for outputs follows similarly.
\end{proof}
\end{fullproof}}
\end{comment}

We shall shortly define $!$-boxes as certain kinds of open subgraphs, and note
that openness is important to preserve the property of being a string graph
(i.e., no branching wires) when $!$-boxes are copied. The following proposition
justifies the use of the topological term ``open''.

\begin{proposition}
	\label{prop:open-subgraph-props}
	If $O,O' \subseteq G$ are open subgraphs, and $H \subseteq G$ is an
	arbitrary subgraph, then $O \cap O'$ and $O \cup O'$ are open in $G$
	and $H \cap O$ is open in $H$.
\end{proposition}
\nulp{\begin{fullproof}
\begin{proof}
	The cases for inputs and outpus are symmetric in all cases, so we
	will only consider inputs.

	Let $i \in \In(G\graphminus (O\cap O'))$.  Clearly $i$ cannot be in both
	$O$ and $O'$; WLOG suppose it is not in $O$.  Suppose $i$ has an
	in-edge $e$ in $G\graphminus O$.  Then the source of $e$ cannot be in
	$O$, and hence not in $O\cap O'$.  But then $e$ is in $G\graphminus
	(O\cap O')$, which contradicts our original assumption about $i$.  So
	$i \in \In(G\graphminus O)$, and therefore in $\In(G)$ by openness of
	$O$.  So $O\cap O'$ is open in $G$.

	Let $i \in \In(G\graphminus (O\cup O'))$.  Clearly $i$ is neither in $O$
	nor $O'$.  Suppose $i \notin \In(G)$, and so it has an in-edge $e$ in
	$G$.  Then the source of $e$ must be in either $O$ or $O'$, since
	otherwise $e$ would be in $G \graphminus (O\cup O')$.  WLOG, suppose it
	is in $O$.  Then $i$ must be in $\In(G\graphminus O)$.  But $O$ is open,
	so $i \in \In(G)$, which is a contradiction.  So $i \in \In(G)$ and
	$O\cup O'$ is open in $G$.

	Let $i \in \In(H\graphminus (H\cap O))$.  Then $i$ is in $G\graphminus O$.
	If $i$ is an input in $G\graphminus O$, $i$ is in $\In(G)$, since $O$ is
	open in $G$, and hence in $\In(H)$, as $i$ is in H.  Otherwise,
	$i$ has an in-edge $e$ in $G\graphminus O$.  The source of $e$ cannot be
	in $O$.  But then it also cannot be in $H$, since otherwise it would
	be in $H\graphminus (H\cap O)$, which would contradict our original
	assumption about $i$.  So $i \in \In(H)$, and $H\cap O$ is open in
	$H$.
\end{proof}
\end{fullproof}}

We encode $!$-boxes into the graph structure itself, by introducing a third vertex type, called a \textit{$!$-vertex}. The extended typegraph $\mathcal G_3$ looks like this:
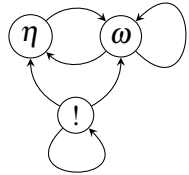
\ctikzfig{kissinger_typegraph_g3}

Note that the typegraph enforces that $!$-vertices can only have out-edges or
edges coming from other $!$-vertices. For a $\mathcal G_3$-typed graph
$(G,\tau)$, we write $\eta(G)$, $\omega(G)$, and $!(G)$ as shorthand for the
preimages $\tau^{-1}(\eta)$, $\tau^{-1}(\omega)$, and $\tau^{-1}(!)$
respectively.  We alter the definition of an \emph{input} slightly from the
string-graph case, due to the new vertex type: a wire-vertex is an input if
the only in-edges are from $!$-vertices.

For a $!$-vertex $b \in !(G)$, let $B(b)$ be its associated $!$-box. This is
the full subgraph whose vertices are the set $\textrm{succ}(b)$ of all of the
successors of $b$. We also define the parent graph of a $!$-vertex
$B^\uparrow(b)$ as the full subgraph of predecessors, that is,
the full subgraph generated by $\textrm{pred}(b)$.

\begin{definition}
	A $\mathcal G_3$-typed graph $G$ is called a \textit{pattern graph} if:
	\begin{enumerate}
		\item the full subgraph with vertices $\eta(G) \cup
		\omega(G)$, denoted $\Sigma(G)$, is a string graph,
		\item the full subgraph with vertices $!(G)$, denoted
			$\beta(G)$, is posetal,
		\item for all $b \in\, !(G)$, $B(b)$ is an open subgraph of $G$, and
		\item for all $b,b'\! \in\, !(G)$, if $b'\! \in B(b)$ then $B(b') \subseteq B(b)$.
	\end{enumerate}
	Let $\catSPatGraph$ be the full subcategory of $\catGraph/\mathcal G_3$ whose objects are pattern graphs.
\end{definition}

Recall that a graph is posetal if it is simple (at most one edge between any
two vertices) and, when considered as a relation, forms a partial order. Note
in particular that this implies $b \in B(b)$ (and $B^\uparrow(b)$), by
reflexivity. This partial order allows $!$-boxes to be nested inside each
other, provided that the subgraph defined by a nested $!$-vertex is totally
contained in the subgraph defined by its parent (condition 4).

We extend the $\Sigma(G)$ and $\beta(G)$ notation to morphisms of
$\catSPatGraph$ by making their operation be the obvious restrictions.  Thus
$\Sigma$ and $\beta$ can be viewed as functors on $\catSPatGraph$.

\begin{definition}
	A pattern graph with no $!$-vertices is called a \textit{concrete
	graph}.
\end{definition}

Note that the full subcategory of $\catSPatGraph$ consisting of concrete
graphs is isomorphic to $\catSGraph$, and there is an obvious canonical
isomorphism.  Concrete graphs and string graphs will therefore be considered
interchangable.

We introduce special notation for pattern graphs. $!$-vertices are drawn as
squares, but rather than drawing edges to all of the node-vertices and
wire-vertices in $B(b)$, we simply draw a box around it.
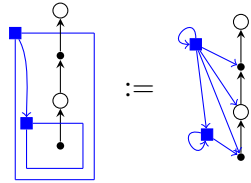
\ctikzfig{kissinger_nested_def}

In this notation, we retain edges between distinct $!$-vertices to indicate
which $!$-boxes are nested as opposed to simply overlapping. This distinction
is important, as nested $!$-boxes are copied whenever their parent is copied.
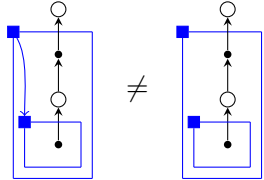
\ctikzfig{kissinger_nested_v_overlap}

In particular, every object in $\catSGraph$ can be considered as a pattern graph that has no $!$-vertices. This embedding $E: \catSGraph \hookrightarrow \catSPatGraph$ is full and coreflective. Its right adjoint is given by the forgetful functor $U : \catSPatGraph \rightarrow \catSGraph$ that drops all of the $!$-boxes.

\subsection{Instantiation} 
\label{sec:pg-instantiation}

Following the ``bang'' operation from linear logic, $!$-boxes admit 4 operations.
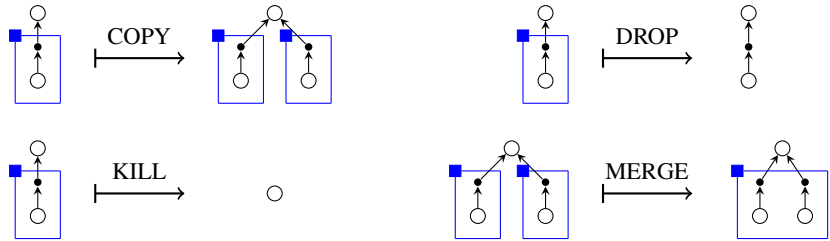
\ctikzfig{kissinger_four_ops}

\begin{definitions}\label{def:bbox-ops}
	For $G$ a pattern graph, and $b,b'\! \in\, !(G)$ where
	$B^\uparrow(b)\graphminus b = B^\uparrow(b')\graphminus b'$ and $B(b) \cap
	B(b') = \{ \}$, the four $!$-box operations are defined as follows:
	\begin{description}
		\item $\COPY_b(G)$ is defined by a pushout of inclusions in $\catGraph/\mathcal G_3$:
		\begin{equation}
			\label{eq:copy-pushout}
			\posquare{G\graphminus B(b)}{G}{G}{\COPY_b(G)}{}{}{}{}
		\end{equation}
		\item $\DROP_b(G) := G\graphminus b$.
		\item $\KILL_b(G) := G\graphminus B(b)$.
		\item $\MERGE_{b,b'}(G)$ is a quotient of $G$ where
		$B^\uparrow(b)$ and $B^\uparrow(b')$ are identified. More
		explicitly, this is the coequaliser
		\begin{equation}
			\label{eq:merge-diagram}
			\begin{tikzpicture}
			    \matrix(m)[cdiag]{
			    B^\uparrow(b) & G & \MERGE_{b,b'}(G') \\};
			    \path [arrs] (m-1-1) edge [arrow above] node {$\hat b$} (m-1-2)
			                 (m-1-1) edge [arrow below] node [swap] {$\hat b'$} (m-1-2)
			                 (m-1-2) edge (m-1-3);
			\end{tikzpicture}
		\end{equation}
		in $\catGraph/\mathcal G_3$ where $\hat b$ is the normal
		inclusion map and $\hat b'$ is the inclusion of
		$B^\uparrow(b')$ into $G$ composed with the obvious
		isomorphism from $B^\uparrow(b)$ to $B^\uparrow(b')$.
	\end{description}
\end{definitions}

Note that all of these operations preserve the property of being a pattern graph.

\begin{theorem}
	\label{thm:instantiation}
	Let $G$ be a pattern graph and $b \in\, !(G)$.  Then the $\mathcal
	G_3$-typed graphs $\COPY_b(G)$, $\DROP_b(G)$ and $\KILL_b(G)$ are all
	pattern graphs.  If we further suppose that $b' \in\, !(G)$ with
	$B^\uparrow(b)\graphminus b = B^\uparrow(b')\graphminus b'$ and $B(b) \cap
	B(b') = \emptygraph$, then $\MERGE_{b,b'}(G)$ is also a pattern graph.
\end{theorem}
\nulp{\begin{fullproof}
\begin{proof}
	The $\COPY$ and $\MERGE$ cases are given by the above propositions.

	Let $H = \DROP_b(G)$.  $G\graphminus \{b\}$ is trivially full in $G$,
	and so $H$ is a pattern graph by lemma
	\ref{lemma:pattgraph-full-subgraph}.

	Similarly, $\KILL_b(G)$ is a pattern graph by the same lemma.
\end{proof}
\end{fullproof}}

Applying one of these four operations any number of times to a pattern graph yields a more specific pattern. As such, we can define a refinement (pre-)ordering on pattern graphs.

\begin{definition}
	For pattern graphs $G$, $H$, we let $G \succeq H$ if and only if $H$
	can be obtained from $G$ (up to isomorphism) by applying the four
	operations from definition~\ref{def:bbox-ops} zero or more times. If
	$H$ is a concrete graph, it is called an \textit{instance} of $G$, and
	the sequence of operations used to obtain $H$ from $G$ is called the
	\textit{instantiation}.
\end{definition}

\subsection{Nested and overlapping \texorpdfstring{$!$}{!}-boxes} 
\label{sec:nesting-overlapping}

Due to the definition of $\COPY$ as a pushout of inclusions, the absence of an
edge between $!$-vertices $b_1$ and $b_2$ with $B(b_1)\cap
B(b_2)\neq\emptygraph$
results in both copies of the contents of $b_1$ created having the same
connectivity to $b_2$ as they had in the original graph:
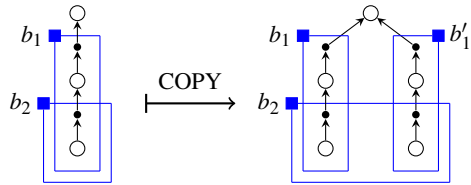
\ctikzfig{kissinger_copy_overlapping2}

Note that it is not actually necessary that $B(b_2)\graphminus b_2$ is
completely contained in $B(b_1)\graphminus b_1$ here.  On the other hand, if
$B(b_2)\graphminus b_2$ \emph{is} a subgraph of $B(b_1)\graphminus b_1$, we could
also add an edge from $b_1$ to $b_2$, which would result in a new copy of
$b_2$ being created to contain the copies of the vertices in $B(b_2)$.
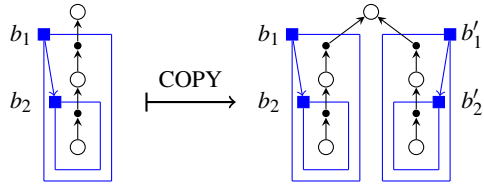
\ctikzfig{kissinger_copy_nested}

\begin{definition}
	For a pattern graph $G$ with distinct $!$-vertices $b_1$ and $b_2$, we
	say $b_2$ is \emph{nested} in $b_1$ if there exists a directed edge
	from $b_1$ to $b_2$. If this is not the case, but $B(b_1)\cap B(b_2)
	\not=\emptygraph$, we call $b_1$ and $b_2$ \emph{overlapping}.
\end{definition}

Both of the above examples could be seen as attempts to formalise the family of all trees of height up to 2. However,
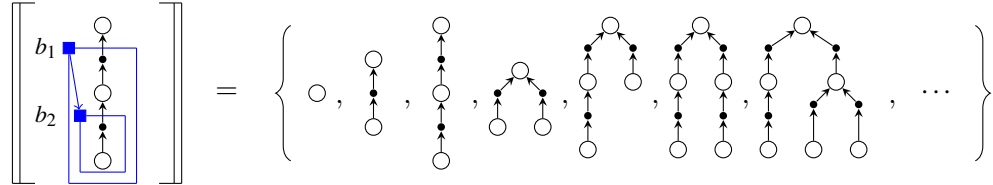
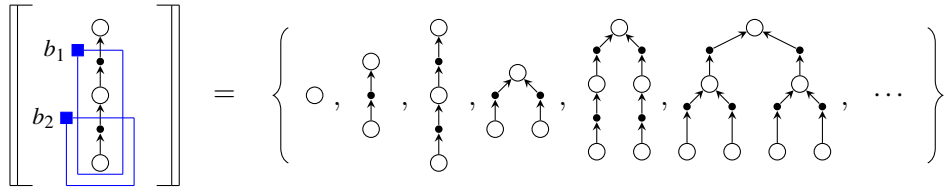
\ctikzfig{kissinger_arbitrary-tree-expand}
but
\ctikzfig{kissinger_balanced-tree-expand}
The absence of nesting restricts the instances to those trees where all the
first-level nodes have the same number of children; in other words, it allows
only balanced trees.  Removing the nesting enforces a higher degree of
regularity in the concrete graphs that can be expressed.

Nesting, in fact, always makes a pattern graph more general in the following
sense:

\begin{proposition}
	Let $G$ be a pattern graph and $b_2$ be nested in $b_1$ in $G$, with the edge
	from $b_1$ to $b_2$ being $e$. Then the set of instances of the graph
	$H=G\graphminus e$
	is a subset of the set of instances of $G$.
\end{proposition}
\nulp{\begin{fullproof}
\begin{proof}
\TODOinline{write a proof}
\end{proof}
\end{fullproof}}

This becomes evident when we observe that we can track operations on $H$ in $G$ by performing a $\MERGE_{b_2,b_2'}$ on the two copies of $b_2$ produced whenever $b_1$ or a copy of it is copied (and performing the same operation otherwise), producing the same pattern graph apart from additional copies of $e$, which must eventually be dropped to obtain a concrete graph.

\section{Matching and rewriting with pattern graphs} 
\label{sec:rw-pattern-graphs}

For those familiar with patterns in functional programming languages, the name
``pattern graph'' suggests that there should be a concept of
\textit{matching}, and given a pattern graph and a string graph, it should be
possible to determine whether the string graph is matched by the pattern
graph.  This is, in fact, the case.  First, we recall how matching between
string graphs is defined.

\begin{definition}
	\label{def:sgraph-matching}
	A monomorphism, $m : G \rightarrow H$, of string graphs is called a
	\emph{string graph matching} when, for every node-vertex $n \in \eta(G)$, the edge
	function of $m$ restricts to a bijection between the set of edges
	connected to $n$ in $G$ and the set of edges connected to $m(n)$ in $H$.
	In this case, $G$ is said to \emph{match} $H$ at $m$.
\end{definition}

The concept of a matching from a pattern graph to a string graph is
straightforward: if there is an instance of the pattern graph that matches the
string graph, then the pattern graph is said to match the string graph.

\begin{definition}
	\label{def:pgraph-matching}
	Let $P$ be a pattern graph, and $H$ a string graph.  If there is an
	instance $G$ of $P$, with instantiation $S$, that matches $H$ at a
	morphism $m$, $P$ is said to \textit{match} $H$ at $m$ under
	instantiation $S$.
\end{definition}

Determining whether such an $m$ and $S$ exist, and what possible values they
can take, is decidable, although we do not have space to show that here.  The
full details are set out in a document in the
Quantomatic\footnote{\url{http://sites.google.com/site/quantomatic}}
repository.

Given a concept of matching, we can proceed to define how to do rewriting of
string graphs using rules built from pattern graphs.  We start by recalling
how rewriting of string graphs using string graph rewrite rules works.

\begin{definition}[Rewrite Rule] \label{def:rewrite-rule}
  A span of string graphs $L \overset{i_1}{\longleftarrow} I \overset{i_2}{\longrightarrow} R$ is called a \textit{rewrite rule}, written $L \rewritesto R$, if
	\begin{enumerate}
		\item $I$ is a point graph and $i_1$ restricts to a bijection $I \cong \Bound(L)$ and $i_2$ to $I \cong \Bound(R)$ and
		\item for all $p \in I$, $i_1(p) \in \In(L) \Leftrightarrow i_2(p) \in \In(R)$ and $i_1(p) \in \Out(L) \Leftrightarrow i_2(p) \in \Out(R)$
	\end{enumerate}
	In other words, $L$ and $R$ \textit{share the same boundary}.
\end{definition}

For a pair of morphisms $\crun{I}{i_1}{L}{m}{G}$, a \textit{pushout complement} is some string graph $G -_{m} L$ completing the pushout square:
\begin{equation}\label{eq:po-complement}
	\posquare{I}{L}{G -_m L}{G}{}{}{}{}
\end{equation}

\begin{theorem}[Dixon-Kissinger~\cite{DixonKissinger2010}]
	For a rewrite rule $L \rewritesto R$ and a matching $m : L \rightarrow
	G$, the pushout complement \eqref{eq:po-complement} exists and is unique.
\end{theorem}

Rewriting is performed via the double-pushout (DPO) technique. First, the pushout complement is computed, to remove the LHS of a rewrite rule, then the RHS is ``glued in'' with a second pushout. The rewrite rule is said to rewrite $G$ to $G'$ (also written $G \rewritesto G'$, when there is no ambiguity) at a matching $m : L \rightarrow G$, when $G'$ is defined according to the following DPO diagram:
\begin{center}
	\begin{tikzpicture}
		\matrix (m) [cdiag] {
        L & I & R \\
        G & G-_{m}L & G'\\
		};
		\path [arrs]
	  	  	(m-1-2) edge [] node [swap] {$i_1$} (m-1-1)
	  	  	(m-1-2) edge [] node {$i_2$} (m-1-3)

	  	  	(m-1-1) edge[] node [swap] {$m$} (m-2-1)
	  	  	(m-1-2) edge[] node {} (m-2-2)
	  	  	(m-1-3) edge[] node {} (m-2-3)

	  	  	(m-2-2) edge [] node {} (m-2-1)
	  	  	(m-2-2) edge [] node {} (m-2-3);
	\NEbracket{(m-2-1)};
    \NWbracket{(m-2-3)};
	\end{tikzpicture}
\end{center}

\begin{definition}
	A \emph{rewrite pattern} is a span of pattern graphs
	$\cspan{L}{i_1}{I}{i_2}{R}$ where
	\begin{enumerate}
		\item $\Sigma(I)$ is a point graph;
		\item $L$ and $R$ share the same boundary via
			$\Sigma(i_1)$ and $\Sigma(i_2)$;
		\item $\beta(i_1)$ and $\beta(i_2)$ are graph isomorphisms;
			and
		\item for each $b \in\, !(I)$, the preimage of $B(i_1(b))$ under $i_1$
			is exactly $B(b)$, and similarly for the preimage of
			$B(i_2(b))$ under $i_2$.
	\end{enumerate}
\end{definition}

Note that the first two conditions ensure that simply applying the forgetful
functor $U:\catSPatGraph \rightarrow \catSGraph$ to this span yields a rewrite
rule, as defined above.

Since our concept of matching involves applying $!$-box operations to the
pattern graph, we need to extend the !-box operations to rewrite patterns.
The rule is that any operation performed on a !-box in $L$ must also be
performed on the equivalent !-box (determined by the bijection induced by
$i_1$ and $i_2$) in $R$.

\begin{lemma}
	\label{lemma:rw-patt-bbox-containment}
	If $\cspan{L}{i_1}{I}{i_2}{R}$ is a rewrite pattern then, for all $b
	\in\, !(I)$, the image of $I\graphminus B(b)$ under $i_1$ is contained
	in $L\graphminus B(i_1(b))$, and similarly for $i_2$ and $R\graphminus
	B(i_2(b))$.
\end{lemma}
\nulp{\begin{fullproof}
\begin{proof}
	Let $v \in i_1[I\graphminus B(b)]$.  Then there is a $v' \in I\graphminus
	B(b)$ such that $v' = i_1(v)$.  $v' \notin B(b)$, so $v' \notin
	i_1^{-1}[B(i_1(b))]$, by requirements of a rewrite pattern.  So $v
	\notin B(i_1(b))$ and hence $v \in L\graphminus B(i_1(b))$.

	The $i_2$ case follows similarly.
\end{proof}
\end{fullproof}}

\nulp{\begin{fullproof}
\begin{lemma}
	\label{lemma:bbox-preimages}
	Let $f : G \rightarrow H$ be a pattern graph morphism, and let $b \in
	!(G)$.  Then $B(b)$ is contained in the preimage of $B(f(b))$ under
	$f$.
\end{lemma}
\begin{proof}
	Let $v \in B(b)$ and let $e$ be the edge in $G$ from $b$ to $v$.  Then
	$f(e)$ is an edge in $H$ from $f(b)$ to $f(v)$, so $f(v)$ is in
	$B(f(b))$.  So $v$ is in the preimage of $B(f(b))$ under $f$.
\end{proof}
\end{fullproof}}

\begin{definition}
	Let $L \rewritesto R$ be a rewrite pattern defined by the span
	$\cspan{L}{i_1}{I}{i_2}{R}$. Let $b,b' \in !(I)$ be mergable
	$!$-vertices such that the pairs of $!$-boxes defined by
	$i_1(b),i_1(b') \in L$ and $i_2(b), i_2(b') \in R$ can also be merged.
	The four !-box operations on pattern graphs have the following
	equivalents on rewrite patterns:
	\begin{description}
	\item $\PCOPY_b(L \rewritesto R)$ is defined by:
	\[ \cspan{\COPY_{i_1(b)}(L)}{i_1'}{\COPY_b(I)}{i_2'}{\COPY_{i_2(b)}(R)} \]
	For $\COPY_b(I)$ and $\COPY_{i_1(b)}(L)$ defined by pushouts:
	\begin{center}
		\posquare{I\graphminus B(b)}{I}{I}{\COPY_b(I)}{\iota}{p_1^I}{\iota}{p_2^I}
		\qquad\qquad
		\posquare{L\graphminus B(i_1(b))}{L}{L}{\COPY_{i_1(b)}(L)}{}{p_1^L}{}{p_2^L}
	\end{center}
	the maps $p_1^L$ and $p_2^L$ agree on $L\graphminus B(i_1(b))$. From
	lemma \ref{lemma:rw-patt-bbox-containment}, we can deduce that $p_1^L
	\circ i_1 \circ \iota = p_2^L \circ i_1 \circ \iota$. We then define $i_1'$ as the map induced by
	the pushout along $I\graphminus B(b)$.
	\begin{equation}
	\label{eq:pcopy}
	\begin{tikzpicture}
	    \matrix(m)[cdiag]{
	      I\graphminus B(b) & I         &                  \\
	      I                & COPY_b(I) & L                \\
	                       & L         & COPY_{i_1(b)}(L) \\
	    };
	    \path [arrs] (m-1-1) edge node {$\iota$} (m-1-2)
	                 (m-1-1) edge node [swap] {$\iota$} (m-2-1)
	                 (m-1-2) edge (m-2-2)
	                 (m-2-1) edge (m-2-2)
	                 (m-2-1) edge node [swap] {$i_1$} (m-3-2)
	                 (m-2-2) edge [dashed] node {$i_1'$} (m-3-3)
	                 (m-1-2) edge node {$i_1$} (m-2-3)
	                 (m-3-2) edge node [swap] {$p_1^L$} (m-3-3)
	                 (m-2-3) edge node {$p_2^L$} (m-3-3);
	    \NWbracket{(m-2-2)}
	\end{tikzpicture}
	\end{equation}
	$i_2'$ is defined similarly.
	\item $\PDROP_b(L \rewritesto R)$ is defined by the span:
	\[ \cspan{\DROP_{i_1(b)}(L)}{i_1'}{\DROP_b(I)}{i_2'}{\DROP_{i_2(b)}(R)} \]
	where $i_1'$ and $i_2'$ are the restrictions of $i_1$ and $i_2$ to $\DROP_b(I)$.
	\item $\PKILL_b(L \rewritesto R)$ is defined similarly:
	\[ \cspan{\KILL_{i_1(b)}(L)}{i_1'}{\KILL_b(I)}{i_2'}{\KILL_{i_2(b)}(R)} \]
	where $i_1'$ and $i_2'$ are again restrictions of $i_1$ and $i_2$.
	\item $\PMERGE_{b,b'}(L \rewritesto R)$ is a span:
	\[ \cspan{\MERGE_{i_1(b),i_1(b')}(L)}{i_1'}{\MERGE_{b,b'}(I)}{i_2'}{\MERGE_{i_2(b),i_2(b')}(R)} \]
	The maps $i_1'$ and $i_2'$ are induced by the coequaliser of $\hat b$ and $\hat b'$.
	\[
	\begin{tikzpicture}
	    \matrix(m)[cdiagsm]{
		& B^\uparrow(b) &   \\
		L & I   & R \\
		\MERGE_{i_1(b),i_1(b')}(L) & \MERGE_{b,b'}(I) & \MERGE_{i_2(b),i_2(b')}(R) \\
	    };
	    \path [arrs] (m-1-2) edge [arrow left] node [swap] {$\hat b$} (m-2-2)
	                 (m-1-2) edge [arrow right] node {$\hat b'$} (m-2-2)
	                 (m-2-2) edge node [swap] {$i_1$} (m-2-1)
	                 (m-2-2) edge node {$i_2$} (m-2-3)
	                 (m-3-2) edge [dashed] node [swap] {$i_1'$} (m-3-1)
	                 (m-3-2) edge [dashed] node {$i_2'$} (m-3-3)
	                 (m-2-1) edge (m-3-1)
	                 (m-2-2) edge (m-3-2)
	                 (m-2-3) edge (m-3-3);
	\end{tikzpicture}
	\]
	\end{description}
\end{definition}

\begin{theorem}
	\label{thm:rw-patt-preserved}
	Let $L \rewritesto R$ be a rewrite pattern. Then applying any of the rewrite !-box operations yields another rewrite pattern.
\end{theorem}
\nulp{\begin{fullproof}
\begin{proof}
	This follows immediately from the above propositions (the $\PDROP$
	case is trivial).
\end{proof}
\end{fullproof}}

From this result and the definition of rewrite !-box operations above, we can
see that, given matching of $L$ against a string graph $G$ at $m$ under
instantiation $S$, applying the equivalent instantiation sequence to the
rewrite pattern $L \rewritesto R$ will produce a rewrite rule that can be used
to rewrite $G$ to another string graph $H$.  In this way, a single rewrite
pattern can take the place of an infinite family of rewrite rules.

\section{Conclusions and future work} 
\label{sec:conclusion}

We have presented a construction for expressing graphs with a certain form of
repetitive structure, as might be informally expressed with ellipses.  This
pattern graph construction has been made in the language of typed graphs,
allowing the application of familiar techniques for reasoning about graphs.
We have demonstrated how it can be used to express rules that appear in
graphical calculi for quantum information processing.

We have also demonstrated how pattern graphs can be used to rewrite string
graphs, and hence how they allow infinitary families of rules to be used when
reasoning mechanically about string diagrams.

We already have a piece of software,
Quantomatic\footnote{\url{http://sites.google.com/site/quantomatic}}, that
implements a restricted version of pattern graphs, and we are currently
extending it to leverage nested and overlapping $!$-boxes.  The na\"ive
algorithm for matching is quite inefficient, and there should be some gains to
be made by making use of the inherent graph symmetries that arise from copying
$!$-boxes.

An obvious next step is to explore how pattern graphs can be rewritten
directly using rewrite patterns, which would allow us to reason by rewriting
about infinite families of graphs simultaneously. In particular, the notions
of pattern graph matching and unification could be applied to perform
Knuth-Bendix completion~\cite{knuth-bendix}, which could be used in
combination with rules generated by other automated means (e.g., conjecture
synthesis~\cite{KissingerCosy2012}) to generate new pattern graph rewrite
rules~\cite{aleksThesis}.

Another way this work can be extended is to develop ways to express richer
families of string graphs. Pattern graphs can be thought of as something
akin to regular expressions, sans alternation. What sorts of families
can we express using analogues to full regular, context-free,
or recursive languages? For example, could such a language effectively
represent things like chains of unbounded length?
\ctikzfig{kissinger_chain-graph}


Another question one might ask is how pattern graphs can be applied to
study more general graph rewriting problems, rather than just rewriting
for string graphs. In this case, many of the concepts of this paper,
with the exception of ``open subgraphs'', translate straightforwardly
to arbitrary typed graphs.


\bibliographystyle{eptcs}
\bibliography{kissinger}

\end{document}